\documentclass{article}
\usepackage[utf8]{inputenc}
\usepackage{amsmath, amssymb, enumerate}
\usepackage{amsthm}

\theoremstyle{plain}
\newtheorem{thm}{Theorem}[section]

\theoremstyle{definition}
\newtheorem{rem}[thm]{Remark}
\newtheorem{expl}[thm]{Example}

\DeclareMathOperator{\lcm}{lcm}
\def\Irr{{\rm Irr}}
\def\II{{\rm I}}
\def\ord{{\mathrm{ord}}}
\def\Q{{\mathbb{Q}}}
\def\F{{\mathbb{F}}}
\def\Z{{\mathbb{Z}}}

\usepackage{fancyhdr}
\usepackage{hyperref}

\theoremstyle{plain}

\title{Counting Irreducible Representations of a Finite Abelian Group}
\author{Thomas Breuer\footnote{Lehrstuhl f\"{u}r Algebra und Zahlentheorie, RWTH Aachen University, Pontdriesch 14/16, 52062 Aachen, Germany; \ E-mail: sam@math.rwth-aachen.de.}, Prashun Kumar\footnote{Corresponding author, Dr. B. R. Ambedkar University Delhi, Delhi 110006, India; \ E-mails: prashunkumar.19@stu.aud.ac.in,  prashun07kumar@gmail.com.}
\ and \ Geetha Venkataraman\footnote{ Dr. B. R. Ambedkar University Delhi, Delhi 110006, India; \ E-mails: geetha@aud.ac.in,  geevenkat@gmail.com.}}

\begin{document}

\maketitle

\abstract{%
Let $q$ be a power of a prime $p$,
$G$ be a finite abelian group, where $p \nmid |G|$,
and let $n$ be a positive integer.
In this paper we find a formula for the number of irreducible representations
of $G$ of a given dimension $n$ over the field of order $q$,
up to equivalence, using Brauer characters.
We also provide a formula for such $n$ using the prime decomposition
of the exponent of $G$ and an algorithm to compute the irreducible degrees and their multiplicities.}


\medskip
\noindent
{\small{\bf Mathematics Subject Classification-MSC2020}{:} }
20H30, 20K01, 20K27, 20K30.

\section{Introduction}

In this paper we consider irreducible representations of a finite abelian group
$G$ over a finite field in characteristic coprime to $|G|$.
Several papers and books deal with counting irreducible (modular)
representations of a group $G$ over a given field,
for example see \cite{SDB1956, IR1964, PW2016}.
However, these do not provide a closed formula or give an explicit count.
Neither do they provide a formula of the possible degrees
of these irreducible representations.

Let $G$ and $H$ be abelian groups of order $n$.
In \cite{PW1950}, the authors find all fields ${\F}$ such that
${\F}G \cong {\F}H$ as group algebras.
In the same paper they also show that
$${\F}G = \displaystyle\sum_{d \mid n} a_d {\F}(\zeta_d)$$
where $G$ is an abelian group of order $n$ and ${\F}$ is a field
whose characteristic does not divide $n$.
Here $\zeta_d$ is a primitive $d$-th root of unity and $a_d$ is a
non-negative integer.
While every irreducible ${\F}G$-module occurs in the summation,
an explicit formula for the number of irreducible ${\F}G$-modules
up to isomorphism of a given dimension is not provided
nor is a formula for the $m$ that can appear as the degree of an
irreducible ${\F}G$-module.
Let $p$ and $r$ be distinct primes.
A formula for the number of irreducible representations of an
elementary abelian $r$-group up to equivalence over a finite field
of order $p$ has been given in \cite{GV1993}.

Let $G$ be a finite abelian group and let ${\F}$ be a finite field such that ${\rm char}$ ${\F} \nmid |G|$. The objective of this paper is to provide a formula for the number of irreducible representations of $G$ over ${\F}$ of a given degree $n$ up to equivalence and to give an explicit formula for the total number of irreducible representations of $G$ over ${\F}$ up to equivalence. We use the theory of Brauer characters to establish this. We refer to \cite[chapter 10, page 192]{PW2016} for an introduction to the theory of Brauer characters. Since there is a one-to-one correspondence between the set of all representations of a finite group $G$ having an abelian image and the set of all representations of $G/G'$, our formula also gives the number of irreducible representations of a finite group $G$ having an abelian image.

We find all the degrees of the irreducible representations of $G$ over ${\F}$ and use it to find the formula for the total number of irreducible representations of $G$ over ${\F}$. Further we provide an algorithm to find these irreducible degrees and their multiplicities.

Throughout this article,
$p$ is a prime, $q$ is a power of $p$,
and $\F_q$ is a finite field of order $q$.

\section{Irreducible representations of a finite abelian group}

Let $G$ be a finite abelian group,
and let $\Irr(G)$ denote the absolutely irreducible characters of $G$
over the complex numbers.
Each $\chi \in \Irr(G)$ has degree $1$,
and all values of $\chi$ are complex roots of unity.
Thus the character field ${\Q}(\chi) = {\Q}(\{ \chi(g); g \in G \})$
is a cyclotomic field ${\Q}(\zeta_d)$,
where $\zeta_d$ is a primitive $d$-th complex root of unity.
For each positive integer $d$,
let $\II_d(G) = \{ \chi \in \Irr(G); \Q(\chi) = \Q(\zeta_d) \}$.

By the identification of the (multiplicative) group $\langle \zeta_d \rangle$
with the (additive) group $\Z/d\Z$ of residues modulo $d$,
we can interpret $\II_d(G)$ as the set of epimorphisms from $G$ to $\Z/d\Z$.

The multiplicative order of $q$ modulo $d$,
for two integers $q$ and $d$, is the smallest positive integer $n$
such that $d$ divides $(q^n - 1)$; it is denoted by $\ord(q \bmod d)$.
In particular, we have $\ord(q \bmod 1) = 1$ for all $q$.

\begin{thm} \label{alternative}
  Let $G$ be a finite abelian group of exponent $e$,
  and let $q$ be a prime power coprime to $|G|$.

  Each irreducible representation of $G$ over the field $\F_q$
  with $q$ elements has degree $\ord(q \bmod d)$,
  for some divisor $d$ of $e$.

  The number of irreducible representations of $G$ of degree $n$
  over $\F_q$ is
  \[
     \frac{1}{n} \sum_{d \in D_n} |\II_d(G)|,
  \]
  where
  $D_n = \{d \mid e ; \ \ord(q \bmod d) = n, \ \textrm{$d$ even if $e$ is even}\}$.
\end{thm}

\begin{proof}
The set $\Irr(G)$ is the disjoint union of the sets $\II_d(G)$,
where $d$ runs over all divisors of $e$ if $e$ is odd,
and over all even divisors of $e$ otherwise.
(Note that $\zeta_2 = -1 \in \Q$ holds,
hence we have $\II_d(G) = \II_{2d}(G)$ if $d$ is odd.)

Let $p$ be a prime not dividing $|G|$.
We identify $\Irr(G)$ with the set $\Irr_p(G)$ of absolutely irreducible
Brauer characters of $G$ in characteristic $p$,
and identify our sets $\II_d(G)$ with the corresponding subsets of $\Irr_p(G)$.

We show that for the Brauer character $\chi \in \II_d(G)$,
the $\F_q$-irreducible character that has $\chi$ as a constituent
has degree $\ord(q \bmod d)$.

Let $\chi$ be a $p$-modular Brauer character in $\II_d(G)$.
It can be realized by a representation over a field with $p^m$ elements
if this field contains a primitive $d$-th root of unity,
that is, if $d$ divides $p^m - 1$.
The minimal field in characteristic $p$ over which $\chi$ can be
realized is given by $m = \ord(p \bmod d)$.

Let $q = p^k$, for some positive integer $k$.
The $\F_q$-irreducible character that has $\chi$ as a constituent
arises as an orbit sum under the action of the Galois group of the field
extension $\F_{p^m}/\F_{p^t}$, where $\F_{p^t} = \F_{p^k} \cap \F_{p^m}$,
that is, $t = \gcd(k, m)$.
This implies that the orbit of $\chi$ has length $m/t = \ord(q \bmod d)$.
Note that the Galois action is semiregular,
and that $\II_d(G)$ is closed under this action:
A generator of the Galois group is the map
$\ast q: \II_d \rightarrow \II_d, \chi \mapsto \chi^q$,
where $\chi^q(g) = \chi(g)^q$.

Hence $\II_d(G)$ contributes $|\II_d(G)|/\ord(q \bmod d)$ characters
of degree $\ord(q \bmod d)$ to the set of $\F_q$-irreducibles.
\end{proof}

\begin{rem}
Using the factorization of the exponent $e$ into prime powers,
we can interpret Theorem~\ref{alternative} as follows.
Let $P$ be the set of prime divisors of $|G|$.
Since
\[
   \ord\left(q \bmod \prod_{r \in P} r^{i_r}\right) =
     \lcm\left\{ \ord\left(q \bmod r^{i_r}\right), r \in P \right\},
\]
the degrees that occur are lcm's of values $\ord(q \bmod r^{i_r})$,
for $r \in P$.
In particular, the $r'$-part of each $\ord(q \bmod r^{i_r})$ is given by
$\ord(q \bmod r)$.
\end{rem}

\begin{rem}
It remains to compute the cardinalities $|\II_d|$.
For that, we write $G = \bigoplus_{r \in P} G_r$
where $P$ is the set of prime divisors of $|G|$
and $G_r$ is the Sylow $r$-subgroup of $G$.
We assume that $G_r$ has the structure
$\bigoplus_{j=1}^{n_r} \Z/r^{a(r,j)}\Z$,
where $\Z/m\Z$ denotes a cyclic group of order $m$
and the $a(r,j)$ are positive integers,
$1 \leq a(r, 1) \leq a(r, 2) \leq \cdots \leq a(r, n_r)$.
Let $e$ denote the exponent of $G$.
Then we have $e = \prod_{r \in P} r^{a(r, n_r)}$.

Mapping $\chi \in \Irr(G)$ to the vector $(\chi_{G_r}; r \in P)$
of restrictions to the Sylow $r$-subgroups defines a bijection
from $\Irr(G)$ to the Cartesian product of $\Irr(G_r)$, for $r \in P$.
For $\chi \in \II_d(G)$,
the restriction $\chi_{G_r}$ has character field $\Q(\zeta_{d_r})$,
where $d_r$ is the $r$-part of $d$.
Thus we get $|\II_d(G)| = \prod_{r \in P} |\II_{d_r}(G_r)|$.
Since it is much easier to deal with the $r$-groups $G_r$ than with $G$,
we compute the cardinalities $|\II_{d_r}(G_r)|$.

For nonnegative $l$ such that $r^l$ divides the exponent of $G_r$,
let $N(G_r,l) = \{ g^{r^l}; g \in G_r \}$ be the kernel of the epimorphism
from $G_r$ to its largest factor group of exponent $r^l$.
In particular $N(G_r,0) = G_r$.
The factor group $F(G_r, l) = G_r/N(G_r, l)$ has the structure
$\bigoplus_{j=1}^{r_p} \Z/r^{\min\{l, a(r, j)\}}\Z$.
For $l \geq 1$, we have
\[
  \II_{r^l}(G_r) = \{ \chi \in \Irr(G_r);
                      N(G_r, l) \subseteq \ker(\chi),
                      N(G_r, l-1) \not\subseteq \ker(\chi) \},
\]
except that
$\II_2(G_2) = \{ \chi \in \Irr(G_r); N(G_2, 1) \subseteq \ker(\chi) \}$.
Note that for $\chi \in \II_{r^l}(G_r)$,
we have $\chi(g^{r^l}) = \chi(g)^{r^l} = 1$ for all $g \in G_r$,
and there is an element $g \in G_r$ such that $\chi(g) = \zeta_{r^l}$,
thus $\chi(g^{r^{l-1}}) \not= 1$.
This implies
\[
  |\II_{r^l}(G_r)| = |F(G_r, l)| - |F(G_r, l-1)| =
  \prod_{j=1}^{n_r} r^{\min\{l, a(r, j)\}}
  - \prod_{j=1}^{n_r} r^{\min\{l-1, a(r, j)\}},
\]
except that $|\II_2(G_2)| = 2^{n_2}$.
\end{rem}

\begin{rem}
The fact that the sets $\II_d(G)$ are nonempty for all divisors $d$ of
$\exp(G)$ implies that any two finite abelian groups of the same exponent
have the exact same list of integers that are degrees of
irreducible representations of these groups over a given finite field
whose characteristic is coprime to the common exponent.
The group structure does not matter.
However when we count the number of irreducible representations for each,
of a given possible degree $m$,
then the cardinality of $\II_d(G)$ and hence the structure of the group
comes into play.
\end{rem}

\begin{rem}
The above statements yield the following algorithm
for computing the irreducible degrees and their multiplicities.

The \emph{input} consists of
a power $q$ of a prime $p$
and the description of an abelian group $G$
by a set $P$ of primes and the $a(r, i)$, as defined above,
where $p$ is not in $P$.

The \emph{output} is the sequence
$(n_1, k_1), (n_2, k_2), \ldots, (n_r, k_r)$,
such that the degrees of the irreducible representations of $G$
over the field with $q$ elements are $n_1, n_2, \ldots, n_r$,
and their multiplicities are $k_1, k_2, \ldots, k_r$, respectively.

We proceed as follows.

\begin{enumerate}
\item
  Compute the exponent $e$ of $G$ as $e = \prod_{r \in P} a(r, n_r)$.
\item
  For $r \in P$,
  precompute the cardinalities $|\II_{r^i}(G_r)|$, $1 \leq i \leq a(r, n_r)$.
\item
  For the divisors $d$ of $e$ (or only the even divisors if $2 \in P$),
  compute the orbit length $n = \ord(q \bmod d)$
  of the Galois group in question on $\II_d(G)$,
  compute $|\II_d(G)|$ from the involved $|\II_{d_r}(G_r)|$,
  and record $|\II_d(G)|/n$ irreducibles of degree $n$.
\item
  For each character degree $n$ that occurs,
  sum up the multiplicities of $n$ arising for different divisors $d$.
  Output the degrees and their total multiplicities.
\end{enumerate}
\end{rem}

\begin{expl}
Let $G$ be an elementary abelian group of order $r^m$,
for some prime $r$ not dividing the fixed prime power $q$,
and set $n = \ord(q \bmod r)$.

If $n = 1$ then $\F_q$ contains $r$-th roots of unity,
$D_1 = \{ 1, r \}$, and there are $|\Irr(G)| = |\II_1 \cup \II_r|$
irreducibles of degree $1$ over $\F_q$.

Otherwise we have $D_1 = \{ 1 \}$, $D_n = \{ r \}$,
and there are $|\II_1| = 1$ irreducibles of degree $1$ and
$|\II_r|/n = (r^m-1)/n$ irreducibles of degree $n$ over $\F_q$.

This is a well known result and is given in \cite[Lemma 2.3.4, page-26]{GV1993}.
\end{expl}

\begin{expl}
Let $G = \Z/9\Z \times \Z/5\Z$, a cyclic group of order $45$,
and consider its irreducible representations over finite fields $\F_{2^m}$,
$1 \leq m \leq 12$.
\[
   \begin{array}{ll}
     1, 2, 4^3, 6, 12^2 & m \in \{ 1, 5, 7, 11 \} \\
     1^3, 2^6, 3^2, 6^4 & m \in \{ 2, 10 \} \\
     1, 2^4, 4^9        & m \in \{ 3, 9 \}
   \end{array} \hspace*{2cm}
   \begin{array}{ll}
     1^{15}, 3^{10}     & m \in \{ 4, 8 \} \\
     1^9, 2^{18}        & m \in \{ 6 \} \\
     1^{45}             & m \in \{ 12 \}
   \end{array}
\]
The notation $1^9, 2^{18}$ means that there are $9$ linear characters
and $18$ irreducible characters of degree $2$.
\end{expl}

\section{Acknowledgements}

Thomas Breuer is supported by the German Research Foundation (DFG)
-- Project-ID 286237555 --
within the SFB-TRR 195 Symbolic Tools in Mathematics and their Applications.

\noindent Prashun Kumar would like to acknowledge the UGC-SRF grant ({\emph{identification number}}: 201610088501) which is enabling his doctoral work.


\end{document}